\documentclass[11 pt, onecolumn, draftcls]{IEEEtran}

\usepackage{amsbsy,color,multicol}
\usepackage{floatflt} 

\usepackage{amsmath}
\usepackage{amssymb}
\usepackage{times}
\usepackage{graphicx}
\usepackage{xspace}
\usepackage{paralist} 
\usepackage{setspace} 
\usepackage{xypic}
\xyoption{curve}
\usepackage{latexsym}
\usepackage{amsthm}
\usepackage{ifthen}
\usepackage{caption}
\usepackage{subcaption}
\usepackage{makeidx,multirow,textcomp,longtable,afterpage,setspace}
\usepackage[english]{babel}
\usepackage[latin1]{inputenc}
\usepackage[ruled,vlined]{algorithm2e}
\usepackage{booktabs}
\usepackage{comment}
\usepackage{url}

\newlength\myindent
\newlength\mycolwid

{
\newtheorem{theorem}{Theorem}
\newtheorem{lemma}[theorem]{Lemma}
\newtheorem{definition}[theorem]{Definition}

\newtheorem{corollary}[theorem]{Corollary}

}

\title{\LARGE\bf Preorder Construct on Simple Undirected Graphs}

\author{Augusto Almeida Santos$^{\star}$, Jos\'e M.~F.~Moura$^{\star}$, and Jo\~{a}o Xavier$^\dagger$

\thanks{This work was partially supported by NSF grant $\#$ CCF-$1513936$, Fundac\~ao para a Ci\^encia e a Tecnologia grant [UID/EEA/50009/2013] and part of the work was developed under the grant SFRH/BD/33516/2008 by the Funda\c{c}\~{a}o para a Ci\^{e}ncia e a Tecnologia through the Carnegie Mellon Portugal program.}
\thanks{$^\star$ A. A. Santos (augusto.pt@gmail.com) and J.~M.~F.~Moura (moura@ece.cmu.edu) are with the Dep.~of Electrical and Computer Engineering,
Carnegie Mellon University, Pittsburgh, PA 15213, USA.}
\thanks{$\dagger$ J. Xavier is with the Institute for Systems and Robotics (ISR/IST), LARSyS, Instituto Superior T\'ecnico, Universidade de Lisboa, Av.~Rovisco Pais, Lisboa, Portugal (jxavier@isr.ist.utl.pt).}
}

\begin{document}

\maketitle
\thispagestyle{empty}
\pagestyle{plain}

\begin{abstract}
We construct a novel preorder on the set of nodes of a simple undirected graph. We prove that the preorder (induced by the topology of the graph) is preserved, e.g., by the logistic dynamical system (both in discrete and continuous time). Moreover, the underlying equivalence relation of the preorder corresponds to the coarsest equitable partition (CEP). This will further imply that the logistic dynamical system on a graph preserves its coarsest equitable partition. The results provide a nontrivial invariant set for the logistic and the like dynamical systems, as we show. We note that our construct provides a functional characterization for the CEP as an alternative to the pure set theoretical iterated degree sequences characterization~\cite{fractional}. The construct and results presented might have independent interest for analysis on graphs or qualitative analysis of dynamical systems over networks.
\end{abstract}

\textbf{Keywords.} Preorder; Symmetry; Graphs; Dynamical Systems.

\section{Introduction}

Graphs have become one of the central mathematical abstractions for representing interacting systems across science and engineering~\cite{barabasi2016network}. In such representations, nodes encode entities or components of a complex system while edges capture interactions, direct dependencies, or communication pathways between them. This framework naturally arises in a broad range of domains, including social networks~\cite{Javidi}, biological systems~\cite{Alon,PAMI_Partial}, epidemiology~\cite{Vesp_Complex}, and neuroscience~\cite{Bassett_Network_Neuroscience}.

In neuroscience in particular, graphs provide a natural language to describe the organization of the brain across multiple spatial and temporal scales. Brain regions may be modeled as nodes while structural or functional interactions between regions define the underlying connectivity network. This perspective has fostered important advances in the study of synchronization, information propagation, cognition, and neurological disorders~\cite{Seizure}. Similar graph-based viewpoints also emerge in epidemics, where nodes represent individuals or communities and edges encode possible transmission pathways, leading to dynamical systems evolving over networks.

Beyond the analysis of a known graph, a major challenge in modern network science concerns the recovery of latent graph structure from observations of dynamics or data. Problems such as network inference, graphical model estimation, or causal discovery aim at uncovering hidden interactions from partial and noisy measurements~\cite{MaterassiSalapakaCDC2015,tomo_isit,SMachado,PAMI_Partial}. In many situations, however, the exact graph structure is either inaccessible or only partially identifiable, motivating the search for structural properties or descriptors and symmetries that remain observable at a more macroscopic level.

One particularly important notion in this context is that of graph symmetry. Classical graph automorphisms capture exact structural invariances and induce equivalence classes of nodes with indistinguishable topological roles. Such symmetries are deeply connected to the behavior of dynamical systems over networks, often leading to synchronization phenomena and dimensionality reduction. However, automorphism-based symmetries may be overly restrictive, failing to capture more subtle forms of structural equivalence that nonetheless produce identical dynamical behavior.

A broader and highly relevant notion is given by equitable partitions, and in particular by the coarsest equitable partition (CEP), which groups together nodes having equivalent neighborhood statistics recursively across the graph. Equitable partitions play an important role in graph reduction, spectral graph theory, synchronization, and network dynamics. Unlike automorphism orbits, equitable partitions can reveal hidden structural regularities even in graphs with trivial automorphism groups.

In this work, we construct a novel preorder on the set of nodes of a graph whose associated equivalence relation coincides with the coarsest equitable partition. The preorder is intrinsically topological and admits an inductive characterization based on injective mappings between neighborhoods and paths of the graph. Moreover, we show that this preorder is preserved by logistic-type dynamical systems evolving over the network. As a consequence, the associated equivalence classes define nontrivial invariant structures for the dynamics.

Our construction provides an alternative functional characterization of the coarsest equitable partition beyond the classical iterated degree-sequence perspective. More broadly, the results establish a connection between graph topology, symmetry, and qualitative properties of nonlinear dynamical systems on networks, potentially bearing independent interest for graph theory, network dynamics, and systems analysis.

\section{Problem Formulation}


Let $G=\left(V, A\right)$ be an undirected graph where~$V$ is the set of~$N< \infty$ nodes and~$A\in \left\{0,1\right\}^{N\times N}$ is the underlying adjacency matrix. Our goal in this document is to construct a preorder on the (finite) set of nodes~$V$ of~$G$ that conveys non-trivial structure: 1) it is preserved by the logistic dynamical system
\begin{equation}
\frac{d}{dt}y_{i}(t) = \left(\sum_{j\sim i}\gamma y_{j}(t)\right)\left(1-y_i(t)\right)-y_i(t)\label{eq:logistic}
\end{equation}
for $i=1,\ldots,N$, in a sense to be made precise; and 2) its underlying equivalence relation coincides with the coarsest equitable partition of a graph. The logistic dynamics consists of a standard set of coupled Ordinary Differential Equations (ODEs) that model the spread of epidemics in networks (e.g.,~\cite{porterdynamical,newman,augusto_qualitative,kurtzi,Melanie}). The parameter~$\gamma$ stands for the rate of infection of a virus in the network; the state variable~$y_i(t)$ models the likelihood of infection of an individual~$i$ at time~$t$, due to transmission from its peers. The state-variable~$y_i(t)$ can be also interpreted as the fraction of individuals infected in a community (or sub-graph click\footnote{In a click, all nodes are connected to all other nodes.})~$i$ in a network of interconnected communities (a.k.a., super-network). The sum in equation~\eqref{eq:logistic} runs over the neighbors of~$i$ and that is what accounts for the underlying network structure in the dynamics.


A preorder on the set of nodes is a binary relation $\succeq$ on $V$ that is reflexive and transitive, i.e.,

$\bullet$ $i \succeq i$ (reflexive)

$\bullet$ $i \succeq i' \succeq i'' \Rightarrow i\succeq i''$ (transitive).

Any preorder on a set $V$ naturally induces a quotient on $V$ given by the following equivalence relation
\begin{equation}
i\cong j \overset{{\sf def.}}\Longleftrightarrow i\succeq j \mbox{ and } j\succeq i.\label{eq:equivale}
\end{equation}
Since `$\succeq$' is not a partial order, the right hand side of~\eqref{eq:equivale} does not imply equality~$i=j$ (up to isomorphism), but a weaker equivalence relation. In fact, we will show that our preorder construct induces (in light of equation~\eqref{eq:equivale}) the so-called coarsest equitable partition (CEP), which is an important coloring or symmetry on a graph and that we will introduce momentarily.

It is clear that the logistic dynamical system preserves the automorphisms of a graph -- in that, if isomorphic nodes are initialized with the same initial condition, then their state evolve in synchrony -- but it is non-trivial to show that it preserves the CEP of the graph (this will follow as a corollary to the fact that it preserves the preorder~`$\succeq$' constructed in this paper).

\textbf{Outline of the paper.} Section~\ref{sec:symmetry} introduces and contrast the definitions of isomorphism and the coarsest equitable partition on graphs; Section~\ref{sec:construct} constructs the preorder; Section~\ref{sec:preserva} shows the implications of the preorder on dynamical systems such as the logistic system, in particular, the logistic dynamical system preserves the preorder and hence admits a non-trivial invariant set; Section~\ref{sec:conclusion} concludes the paper.

\vspace*{2ex}

\textbf{Preliminary Notation.}

\vspace*{2ex}

\begin{itemize}
  \item $\mathbb{N}=\left\{1,2,3,\ldots\right\}$;
  \item $V(G)=$ set of nodes of graph~$G$ -- sometimes denoted simply by~$V$;
  \item $\mathcal{N}(i)=$ set of neighbors to the node $i$ (also named neighborhood to $i$) in the undirected graph $G=\left(V,E\right)$;
  \item $d(i)=$ degree of the node $i$ in the undirected graph $G=\left(V,E\right)$;
  \item $\mathbf{d}(\mathbf{p})=$ vector collecting the degrees of the nodes stacked in the vector $\mathbf{p}\in V^{\ell}$, assuming that $G=\left(V,E\right)$ is undirected;
  \item $\pi_{I}^{\ell}\,:\,V^{\ell+1}\rightarrow V^{\left|I\right|}:$ is the canonical projection on the coordinates indexed by the set~$I\subset \left\{1,\ldots,\ell\right\}$.
\end{itemize}

\section{Symmetries on Graphs}\label{sec:symmetry}

In this Section, we define some notions of symmetry on graphs and present the exact notion that fits our purposes, namely, the coarsest equitable partition (CEP). We start by introducing the definition of isomorphism on graphs.

%
%

\begin{definition}[Graph Isomorphism]\label{def:isomorphgraph}
We say that two graphs~$\mathcal{G}_1=\left(V_1,E_1\right)$ and~$\mathcal{G}_2=\left(V_2,E_2\right)$ are isomorphic, and represent it by~$\mathcal{G}_1\cong \mathcal{G}_2$, whenever there exists a bijection~$f\,:\,V_1\rightarrow V_2$ with
\begin{equation}
i\sim j \Leftrightarrow f(i)\sim f(j)\nonumber
\end{equation}
where~`$i \sim j$' means that~$i$ and~$j$ are neighbors in the graph.
\end{definition}

\begin{definition}[Graph Automorphism]\label{def:automorphgraph}
We call~$f\,:V\,\rightarrow V$ an automorphism on~$G=\left(V,E\right)$ if~$f$ is an isomorphism from~$G$ onto~$G$, i.e.,~$f$ is bijective and
\begin{equation}
i\sim j \Leftrightarrow f(i)\sim f(j).\nonumber
\end{equation}
\end{definition}
We say that two nodes~$i$ and~$j$ in a graph are isomorphic, and denote it as~$i\approx j$, whenever there exists an automorphism~$f$ with~$f(i)=j$ or~$f(j)=i$. The set of automorphisms on a graph~$G$ endowed with the operation of composition conforms to a group that we refer to as~${\sf Aut}(G)$.

\begin{definition}[Orbits]\label{def:orbit}
Let~$i$ be a node in~$G$. We refer to
\begin{equation}
[i]\overset{\Delta}={\sf Aut}(G).i=\left\{f(i)\,:\,f\in {\sf Aut}(G)\right\}\nonumber
\end{equation}
as the orbit of the vertex~$i$. In words, two nodes are in the same orbit if and only if they are isomorphic.
\end{definition}

The orbits of a graph in definition~\ref{def:orbit} define an equivalence relation between nodes which allows to quotient the set of nodes into equivalence classes
\begin{equation}
\left[i\right]\overset{\Delta}=\left\{j\in V\,:\, j\approx i\right\}.\nonumber
\end{equation}
Fig.~\ref{fig:quotient} illustrates the chromatic partition of a graph according to this equivalence relation. Monochromatic graphs under this equivalence relation are referred to as vertex-transitive graphs.
\begin{figure} [hbt]
\begin{center}
\includegraphics[scale= 0.5]{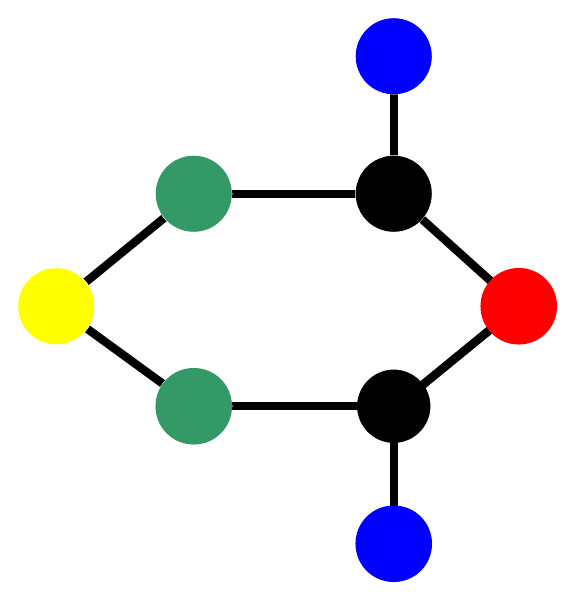}
\caption{Nodes with the same color are isomorphic nodes.}\label{fig:quotient}
\end{center}
\end{figure}
It is natural to expect that the partition illustrated in Fig.~\ref{fig:quotient} is preserved by the logistic dynamical system in the sense that if all nodes have the same initial condition, i.e.,~$\mathbf{y}(0)=\mathbf{1}_N y_0$, then, there is no reason to expect that eventually, say one of the indigo nodes will increase its state (e.g., degree of infection in an epidemics) faster than the other indigo node. Indeed, and more formally, we can show that such dynamical system preserves the partition induced by the orbits of the automorphism group of a graph. In particular, the underlying dynamical system admits a lower dimensional version quotiented by the automorphism symmetries of the graph.

Fig.~\ref{fig:frucht} illustrates a regular graph with degree~$3$, so-called Frucht graph, that entails some asymmetry. For instance, its group of automorphisms is the trivial one (for more details, refer to~\cite{implementing}), i.e., no two nodes are isomorphic in the Frucht graph. In other words, the Frucht graph is a regular graph that is not vertex-transitive. From such asymmetry, it is not clear whether nodes initialized at the same state will evolve evenly under an epidemics modeled by the logistic dynamics.
\begin{figure} [hbt]
\begin{center}
\includegraphics[scale= 0.4]{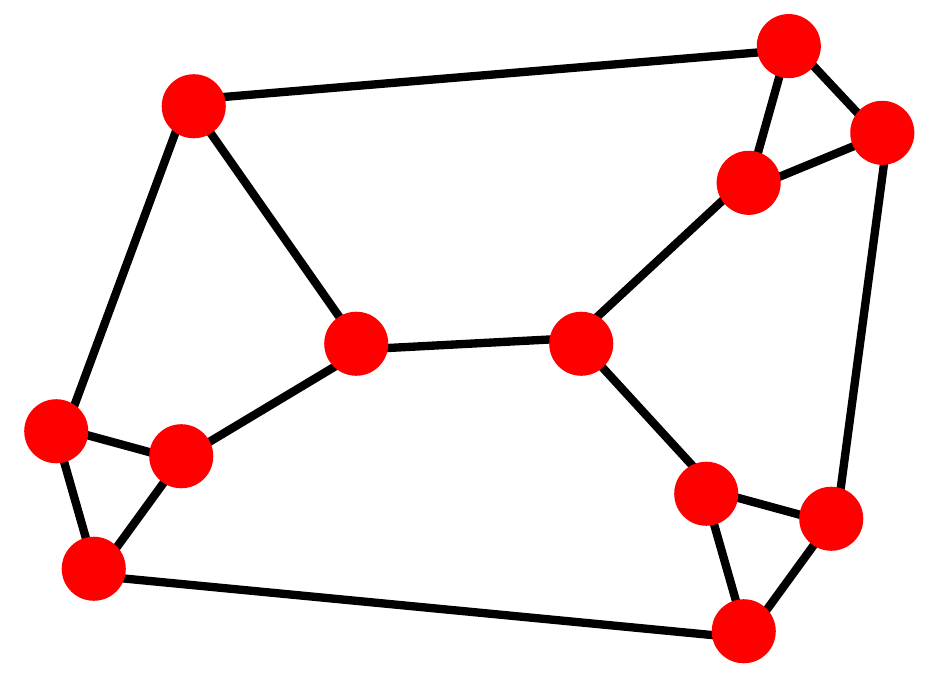}
\caption{Frucht Graph.}\label{fig:frucht}
\end{center}
\end{figure}
Despite the asymmetry, all nodes are equivalent to each other in the CEP (or also in our preorder construct) sense. Therefore, as corollary to the results proved in this paper, nodes departing from the same state, in any regular graph (not necessarily vertex-transitive), will evolve in synchrony -- in particular, a regular network admits a $1$-dimensional dynamics version.

\textbf{Remark.} Note that the synchrony induced by the CEP on the logistic dynamics shall not be taken for granted as one can find counter-examples of dynamical systems whose qualitative properties depend on the nodes even for a regular graph (that is not vertex-transitive), refer, e.g., to~\cite{aldous1989}.

The previous discussion illustrates that even though the orbits of an automorphism are preserved by the logistic and the like dynamical systems, this is not the coarsest coloring that is preserved. Next, we introduce the coarsest equitable partition (for more details refer to~\cite{fractional}). First, let~$\mathcal{P}=\left\{\mathcal{P}_1,\ldots,\mathcal{P}_k\right\}$ be a partition on~$V$. Define
\begin{equation}
s_{\ell}(i) = \left|\mathcal{N}(i)\cap \mathcal{P}_{\ell}\right|\nonumber
\end{equation}
as the number of neighbors of~$i$ in the partition~$\ell$, and define~
\begin{equation}
\mathbf{s}(i)=\left(s_1(i),\ldots,s_{d(i)}(i)\right),\nonumber
\end{equation}
as the distribution of the~$d(i)$ neighbors of~$i$ across the~$\ell$ elements of the partition.

\begin{definition}[Equitable partition]
The partition~$\mathcal{P}$ is called equitable whenever
\begin{equation}
i\approx_{\mathcal{P}} j \Leftrightarrow \mathbf{s}(i)=\mathbf{s}(j).\nonumber
\end{equation}
\end{definition}
In words, two nodes~$i$ and $j$ lie in the same class if and only if they have the same distribution of neighbors across classes. In particular, they have the same degree. The partition induced by the orbits of the group of automorphisms in a graph is an equitable partition, but not the coarsest one in the graph. As referred to before, no two nodes in the Frucht graph in Fig.~\ref{fig:frucht}, which is a regular graph, are equivalent w.r.t. the group of automorphism whereas all nodes are equivalent w.r.t. our preorder construct equivalence, which is equitable and corresponds to the coarsest equitable partition. In certain cases, e.g., when the underlying graph~$G$ is a tree, then the orbits of the automorphism group and the coarsest equitable partition coincide. The next lemma asserts that the \textbf{coarsest} equitable partition exists in any (finite) graph.
\begin{lemma}
Let~$\mathcal{P}$ and~$\mathcal{Q}$ be two equitable partitions on~$G$. Then, there exists the finest equitable partition that is coarser than~$\mathcal{P}$ and~$\mathcal{Q}$ and it is denoted as
\begin{equation}
\mathcal{P}\vee\mathcal{Q}\nonumber
\end{equation}
\end{lemma}

We now present the characterization of the coarsest equitable partition on a graph~$G$, given by the multisets of the so-called iterated degree sequence of the nodes. First, let
\begin{eqnarray}
d^{(1)}(i) & \overset{\Delta}= & \left\{d(v)\,:\,v\in\mathcal{N}(i)\right\}\nonumber\\
d^{(2)}(i) & \overset{\Delta}= & \left\{d^{(1)}(v)\,:\,v\in\mathcal{N}(i)\right\}\nonumber\\
& \vdots & \nonumber\\
d^{(k+1)}(i) & \overset{\Delta}= & \left\{d^{(k)}(v)\,:\,v\in\mathcal{N}(i)\right\}\nonumber
\end{eqnarray}
For instance, from Fig.\ref{fig:quotient} we have
\begin{eqnarray}
d^{(1)}(2) & = & \left\{2,3\right\}\nonumber\\
d^{(2)}(2) & = & \left\{\left\{2,2\right\},\,\left\{1,2,2\right\}\right\}\nonumber\\
\end{eqnarray}

One can extend it to the infinitary sequence, so-called ultimate degree sequence
\begin{equation}
D\left(i\right)\overset{\Delta}=\left(d(i),d^{(1)}(i),d^{(2)}(i),\ldots,d^{(k)}(i),\ldots\right)\nonumber
\end{equation}
Now,~$\mathcal{P}$ is the coarsest equitable partition in a graph~$G$ if and only if
\begin{equation}
i\approx_{\mathcal{P}} j \Leftrightarrow D(i)=D(j).\nonumber
\end{equation}
Observe in particular that all nodes in a regular graph fall in the same class with respect to the coarsest equitable partition.

We rephrase a corollary to Theorem~$6.5.1$ in~\cite{fractional} as follows.
\begin{corollary}
Let~$\mathcal{P}$ be the partition induced on the set of nodes~$V$ by the iterated degree sequence on the graph~$G=\left(V,E\right)$. Let~$\mathcal{Q}$ be an equitable partition on~$V$. Then,~$\mathcal{P}$ is coarser than~$\mathcal{Q}$.
\end{corollary}
There are efficient algorithms to determine the coarsest equitable partition in~$O(N\log(N))$. Determining the finer partition induced by the orbits of the automorphism group is NP-hard. We reserve the symbol~$\approx$ as the equivalence relation associated to the coarsest equitable coloring.

%

\section{Preorder-construct on simple undirected graphs}\label{sec:construct}

In this Section, we construct a preorder on the set of nodes of an undirected graph~$G=\left(V,E\right)$. The idea of the preorder is to formalize the following domination concept: a node~$i$ is greater or dominates a node~$j$ whenever the degree of~$i$ is greater than the degree of~$j$; and the degrees of the neighbors of~$i$ are greater than the degrees of the neighbors of~$j$; and the degrees of the neighbors of the neighbors of~$i$ are greater than the degrees of the neighbors of the neighbors of~$j$ and so forth. This idea must be clearly formalized to make any sense.

Define
\begin{equation}
\mathcal{P}_{\ell}(i)\overset{\Delta}= \left\{\mathbf{p}^{\ell}=\left(i,i_1,\ldots,i_{\ell}\right)\in V^{\ell+1}\,:\,i_{k}\sim i_{k+1}\,\,\forall{k\leq \ell-1}\right\}\nonumber,
\end{equation}
as the set of all paths, on the graph~$G$, departing from the node~$i$ and with length~$\ell$.

Now, we introduce one of the main concepts for the preorder construct.

\begin{definition}[$\ell$-adapted function]
Let~$i,j\in V(G)$ and~$\ell \in \mathbb{N}$. We define the class of $\ell$-adapted functions on~$G$ at the pair~$\left(j,i\right)$ as the set of functions
\begin{equation}
\begin{array}{rcl} f_{ji}^{\ell}: \mathcal{P}_{\ell}(j) & \to & \mathcal{P}_{\ell}(i)\\ \left(j,j_1,\ldots,j_{\ell}\right) & \mapsto & \left(i,i_1,\ldots,i_{\ell}\right) \end{array}\nonumber
\end{equation}
fulfilling the next two properties
\begin{eqnarray}
\pi^{\ell}_{\leq m}\left(\mathbf{p}^{\ell}\right) & = & \pi^{\ell}_{\leq m}\left(\mathbf{\widetilde{p}}^{\ell}\right)\nonumber\\
& \Downarrow & \\
\pi^{\ell}_{\leq m}\left(f^{\ell}_{ji}\left(\mathbf{p}^{\ell}\right)\right) & = & \pi^{\ell}_{\leq m}\left(f^{\ell}_{ji}\left(\mathbf{\widetilde{p}}^{\ell}\right)\right),\nonumber
\end{eqnarray}
and
\begin{eqnarray}
\pi^{\ell}_{\leq m}\left(\mathbf{p}^{\ell}\right) & \neq & \pi^{\ell}_{\leq m}\left(\mathbf{\widetilde{p}}^{\ell}\right)\nonumber\\
& \Downarrow & \label{eq:notequal}\\
\pi^{\ell}_{\leq m}\left(f^{\ell}_{ji}\left(\mathbf{p}^{\ell}\right)\right) & \neq & \pi^{\ell}_{\leq m}\left(f^{\ell}_{ji}\left(\mathbf{\widetilde{p}}^{\ell}\right)\right),\nonumber
\end{eqnarray}
for all~$\mathbf{p}^{\ell},\mathbf{\widetilde{p}}^{\ell}\in \mathcal{P}_{\ell}(j)$, where we defined
\begin{equation}
\pi^{\ell}_{\leq m}\,:\,V^{\ell+1}\longrightarrow V^{m+1}\nonumber
\end{equation}
with~$\pi^{\ell}_{\leq m}\left(x_0,x_1,\ldots,x_{\ell}\right)= \left(x_0,x_1,\ldots,x_{m}\right)$, i.e.,~$\pi_{\leq m}$ returns the first~$m\leq \ell$ coordinates from the input vector.
\end{definition}
In words, if the first~$m$ coordinates of a path~$\mathbf{p}^{\ell}$ coincide with the first~$m$ coordinates of a path~$\mathbf{\widetilde{p}}^{\ell}$, then the same should hold for the images of~$\mathbf{p}^{\ell}$ and~$\mathbf{\widetilde{p}}^{\ell}$ by the map~$f^{\ell}_{ij}$. For simplicity, from now one, we will commit an abuse of notation and we will write~$\pi_{\leq m}$ instead of~$\pi_{\leq m}^{\ell}$.

We assume throughout that~$f^{\ell}_{ii}$ is the identity map for any~$i\in V(G)$, which is clearly an adapted function. We now observe two important properties.

\begin{lemma}[``Group" property]
Let~$g^{\ell}_{jk}$ and~$h^{\ell}_{ki}$ be two $\ell$-adapted functions at~$\left(j,k\right)$ and~$\left(k,i\right)$, respectively. Then,
\begin{equation}
f^{\ell}_{ji}\overset{\Delta}=h^{\ell}_{ki}\circ g^{\ell}_{jk}\nonumber
\end{equation}
is $\ell$-adapted at~$\left(j,i\right)$.
\end{lemma}

\begin{proof}
Indeed, let~$\mathbf{p}^{\ell}, \mathbf{\widetilde{p}}^{\ell}\in \mathcal{P}_{\ell}(j)$ with
\begin{equation}
\pi_{\leq m}\left(\mathbf{p}^{\ell}\right)= \pi_{\leq m}\left(\mathbf{\widetilde{p}}^{\ell}\right)\nonumber
\end{equation}
for some~$m\leq \ell$, then
\begin{equation}
\pi_{\leq m}\left(g^{\ell}_{jk}\left(\mathbf{p}^{\ell}\right)\right)= \pi_{\leq m}\left(g^{\ell}_{jk}\left(\mathbf{\widetilde{p}}^{\ell}\right)\right)\nonumber
\end{equation}
as~$g^{\ell}_{jk}$ is $\ell$-adapted and thus,
\begin{equation}
\pi_{\leq m}\left(f^{\ell}_{jk}\left(\mathbf{p}^{\ell}\right)\right)= \pi_{\leq m}\left(f^{\ell}_{jk}\left(\mathbf{\widetilde{p}}^{\ell}\right)\right)\nonumber
\end{equation}
as~$h^{\ell}_{jk}$ is $\ell$-adapted. Similarly, we can establish the property in equation~\eqref{eq:notequal}.
\end{proof}

\begin{lemma}\label{lem:invertible}
If~$f^{\ell}_{ji}\,:\,\mathcal{P}_{\ell}(j)\rightarrow \mathcal{P}_{\ell}(i)$ is an $\ell$-adapted invertible function, then~$g^{\ell}_{ij}\overset{\Delta}=\left(f^{\ell}_{ji}\right)^{-1}$ is an $\ell$-adapted function.
\end{lemma}

\begin{proof}
Assume that~$g^{\ell}_{ij}$ is not an adapted function. Then, either one of the cases below should hold true
\begin{enumerate}
  \item $\pi_{\leq m}\left(\mathbf{p}^{\ell}\right)= \pi_{\leq m}\left(\mathbf{\widetilde{p}}^{\ell}\right)$, but $\pi_{\leq m}\left(g^{\ell}_{jk}\left(\mathbf{p}^{\ell}\right)\right)\neq \pi_{\leq m}\left(g^{\ell}_{jk}\left(\mathbf{\widetilde{p}}^{\ell}\right)\right)$ for some~$m$ and $\mathbf{p}^{\ell}, \mathbf{\widetilde{p}}^{\ell}\in\mathcal{P}_{\ell}(i)$.
  \item $\pi_{\leq m}\left(\mathbf{p}^{\ell}\right)\neq \pi_{\leq m}\left(\mathbf{\widetilde{p}}^{\ell}\right)$, but $\pi_{\leq m}\left(g^{\ell}_{jk}\left(\mathbf{p}^{\ell}\right)\right)= \pi_{\leq m}\left(g^{\ell}_{jk}\left(\mathbf{\widetilde{p}}^{\ell}\right)\right)$ for some~$m$ and $\mathbf{p}^{\ell}, \mathbf{\widetilde{p}}^{\ell}\in\mathcal{P}_{\ell}(i)$.
\end{enumerate}
Now, case 1) implies that
\begin{eqnarray}
\pi_{\leq m}\left(f^{\ell}_{ji}\left(g^{\ell}_{ij}\left(\mathbf{p}^{\ell}\right)\right)\right) & = & \pi_{\leq m}\left(\mathbf{p}^{\ell}\right)\nonumber\\
& = & \pi_{\leq m}\left(\mathbf{\widetilde{p}}^{\ell}\right)\nonumber\\
& = & \pi_{\leq m}\left(f^{\ell}_{ji}\left(g^{\ell}_{ij}\left(\mathbf{\widetilde{p}}^{\ell}\right)\right)\right)\nonumber
\end{eqnarray}
which is a contradiction as
\begin{equation}
\pi_{\leq m}\left(g^{\ell}_{ij}\left(\mathbf{p}^{\ell}\right)\right)\neq \pi_{\leq m}\left(g^{\ell}_{ij}\left(\mathbf{\widetilde{p}}^{\ell}\right)\right)\nonumber
\end{equation}
and~$f^{\ell}_{ji}$ is an adapted function. Case 2) similarly leads to a contradiction, therefore the inverse map~$g^{\ell}_{ij}$ is an $\ell$-adapted function.
\end{proof}

We now observe that if~$f^{\ell}_{ji}\,:\,\mathcal{P}_{\ell}(j)\rightarrow \mathcal{P}_{\ell}(i)$ is an adapted function, then its restriction to the $(\ell-1)$-paths~$\mathcal{P}_{\ell-1}(j)$, denoted as~$f^{\ell-1}_{ji}$, is an $(\ell-1)$-adapted function. Moreover,~$f^{\ell}_{ji}$ extends~$f^{\ell-1}_{ji}$ in that if~$\mathbf{p}^{\ell-1}=\pi_{\leq \ell-1}\left(\mathbf{p}^{\ell}\right)$, then
\begin{equation}
\pi_{\leq \ell-1}\left(f^{\ell}_{ji}(\mathbf{p}^{\ell})\right)=f^{\ell-1}_{ji}\left(\mathbf{p}^{\ell-1}\right)\label{eq:extends}
\end{equation}
or else, if~$\mathbf{p}^{\ell-1} \neq \pi_{\leq \ell-1}\left(\mathbf{p}^{\ell}\right)$, then
\begin{equation}
\pi_{\leq \ell-1}\left(f^{\ell}_{ji}(\mathbf{p}^{\ell})\right)\neq f^{\ell-1}_{ji}\left(\mathbf{p}^{\ell-1}\right)\label{eq:extends2}
\end{equation}

\begin{definition}[Adapted sequence]
We call
\begin{equation}
\left(f^{\ell}_{ji}\,:\,\mathcal{P}_{\ell}(j)\rightarrow \mathcal{P}_{\ell}(i)\right)_{\ell}\nonumber
\end{equation}
an adapted sequence (on~$\ell$) of adapted functions at the pair~$\left(j,i\right)$ if~$f^{\ell+1}_{ji}$ extends~$f^{\ell}_{ji}$ for any~$\ell \in \mathbb{N}$ in the sense of equations~\eqref{eq:extends}-\eqref{eq:extends2}.
\end{definition}

The following Lemma remarks that if the adapted sequence is comprised of invertible maps, then the inverse maps sequence is an adapted sequence as well.

\begin{lemma}[Inverse sequence]\label{lem:inversesequence}
Let~$\left(f^{\ell}_{ji}\right)_{\ell}$ be an adapted sequence with $f^k_{ji}$ being an invertible $k$-adapted function for all~$k$. Then, the sequence~$\left(g^{\ell}_{ij}\right)_{\ell}$, where~$g^{\ell}_{ij}\overset{\Delta}=\left(f^{\ell}_{ji}\right)^{-1}$, is an adapted sequence.
\end{lemma}

\begin{proof}
Note that each~$g^{\ell}_{ij}$ is an adapted function from Lemma~\ref{lem:invertible}. We are just left to prove that~$g^{\ell+1}_{ij}$ extends~$g^{\ell}_{ij}$ for all~$\ell$, in the sense of equations~\eqref{eq:extends}-\eqref{eq:extends2}. Suppose that~$\left(g^{\ell}_{ij}\right)_{\ell}$ is not an adapted sequence. Then, there exists an~$\ell\in \mathbb{N}$ so that~$g^{\ell+1}_{ij}$ does not extend~$g^{\ell}_{ij}$, or in other words, one of the conditions below hold true
\begin{enumerate}
  \item $\mathbf{p}^{\ell}=\pi_{\leq \ell}\left(\mathbf{p}^{\ell+1}\right)$, but $\pi_{\leq \ell}\left(g^{\ell+1}_{ji}(\mathbf{p}^{\ell+1})\right) \neq g^{\ell}_{ji}\left(\mathbf{p}^{\ell}\right)$;
  \item $\mathbf{p}^{\ell}\neq \pi_{\leq \ell}\left(\mathbf{p}^{\ell+1}\right)$, but $\pi_{\leq \ell}\left(g^{\ell+1}_{ji}(\mathbf{p}^{\ell+1})\right) = g^{\ell}_{ji}\left(\mathbf{p}^{\ell}\right)$.
\end{enumerate}
Define~$\mathbf{\widetilde{p}^{\ell}}\overset{\Delta}=g_{ij}^{\ell}\left(\mathbf{p}^{\ell}\right)$ and~$\mathbf{\widetilde{p}}^{\ell+1}\overset{\Delta}=g_{ij}^{\ell+1}\left(\mathbf{p}^{\ell+1}\right)$. Case 1) implies that
\begin{equation}
f^{\ell}_{ji}\left(\mathbf{\widetilde{p}}^{\ell}\right)\neq \pi_{\leq \ell}\left(f^{\ell+1}_{ji}\left(\mathbf{\widetilde{p}}^{\ell+1}\right)\right),\nonumber
\end{equation}
but~$f^{\ell}_{ji}\left(\mathbf{\widetilde{p}}^{\ell}\right)=\mathbf{p}^{\ell}$ and~$f^{\ell+1}_{ji}\left(\mathbf{\widetilde{p}}^{\ell+1}\right)=\mathbf{p}^{\ell+1}$ which contradicts the assumption. Case 2) leads similarly to a contradiction. Therefore, the inverse sequence~$\left(g^{\ell}_{ij}\right)_{\ell}$ is an adapted sequence.
\end{proof}

\begin{theorem}[``Group" property for adapted sequences]\label{th:groupadaptedseq}
Let~$\left(h^{\ell}_{ki}\right)$ and~$\left(g^{\ell}_{jk}\right)$ be two adapted sequences. Then, the pointwise composition
\begin{equation}
\left(f^{\ell}_{ji}\overset{\Delta}= h^{\ell}_{ki} \circ g^{\ell}_{jk}\right)_{\ell}\nonumber
\end{equation}
is an adapted sequence.
\end{theorem}

\begin{proof}
The sequence~$\left(f^{\ell}_{ji}\right)$ must fulfill the defining conditions in equations~\eqref{eq:extends}-\eqref{eq:extends2}. Let~$\mathbf{p}^{\ell}=\pi_{\leq \ell}\left(\mathbf{p}^{\ell+1}\right)$, with~$\mathbf{p}^{\ell}\in\mathcal{P}_{\ell}(j)$ and~$\mathbf{p}^{\ell+1}\in\mathcal{P}_{\ell+1}(j)$. Then,
\begin{equation}
\pi_{\leq \ell}\left(g^{\ell+1}_{jk}\left(\mathbf{p}^{\ell+1}\right)\right)=g^{\ell}_{jk}\left(\mathbf{p}^{\ell}\right)\label{eq:g}
\end{equation}
as~$\left(g^{\ell}_{jk}\right)_{\ell}$ is an adapted sequence. Moreover, define~$\mathbf{\widetilde{p}}^{\ell+1}\overset{\Delta}= g^{\ell+1}_{jk}\left(\mathbf{p}^{\ell+1}\right)$ and~$\mathbf{\widetilde{p}}^{\ell}\overset{\Delta}= g^{\ell}_{jk}\left(\mathbf{p}^{\ell}\right)$. Thus, equation~\eqref{eq:g} can be written as
\begin{equation}
\pi_{\leq \ell}\left(\mathbf{\widetilde{p}}^{\ell+1}\right)=\mathbf{\widetilde{p}}^{\ell}\nonumber
\end{equation}
which implies that
\begin{equation}
\pi_{\leq \ell}\left(h^{\ell+1}_{jk}\left(\mathbf{\widetilde{p}}^{\ell+1}\right)\right)=h^{\ell}_{jk}\left(\mathbf{\widetilde{p}}^{\ell}\right)\nonumber
\end{equation}
and hence,
\begin{equation}
\pi_{\leq \ell}\left(f^{\ell+1}_{ji}\left(\mathbf{p}^{\ell+1}\right)\right)=f^{\ell}_{ji}\left(\mathbf{p}^{\ell}\right).\nonumber
\end{equation}
The second case in equation~\eqref{eq:extends2} holds similarly and~$\left(f^{\ell}_{ji}\right)$ is an adapted sequence.
\end{proof}

Now, we introduce the main construct of this document.

\begin{definition} [Preorder]\label{def:preorder}
Let~$i,j\in V(G)$. We say that~$i\succeq j$, whenever there exists an adapted sequence
\begin{equation}
\left(f^{\ell}_{ji}\,:\,\mathcal{P}_{\ell}(j)\rightarrow \mathcal{P}_{\ell}(i)\right)_{\ell}\nonumber
\end{equation}
where both properties hold
\begin{itemize}
  \item $f^{\ell}_{ji}\,:\,\mathcal{P}_{\ell}(j)\rightarrow \mathcal{P}_{\ell}(i)$ is injective;
  \item $\mathbf{d}\left(f^{\ell}_{ji}\left(\mathbf{p}^{\ell}\right)\right)\geq \mathbf{d}\left(\mathbf{p}^{\ell}\right)$
\end{itemize}
for all~$\mathbf{p}^{\ell}\in\mathcal{P}_{\ell}(j)$ and all~$\ell \in\mathbb{N}$.
\end{definition}
We say that an adapted sequence~$\left(f^{\ell}_{ji}\right)$ \textbf{explains} the inequality~$i\succeq j$, whenever~$\left(f^{\ell}_{ji}\right)$ fulfills the two conditions in the definition~\ref{def:preorder}. We may also refer that $i\succeq j$ is \textbf{explained} by~$\left(f^{\ell}_{ji}\right)$ under these conditions.

\begin{theorem}
The binary relation~$\succeq$ in definition~\ref{def:preorder} is a preorder on the set of nodes~$V(G)$ of the graph~$G$, i.e., it is reflexive and transitive:
\begin{itemize}
  \item \textbf{(Reflexivity)}   $i\succeq i$;
  \item \textbf{(Transitivity)} $i \succeq k \succeq j \Rightarrow i \succeq j$.
\end{itemize}
\end{theorem}

\begin{proof}
\underline{\textbf{Reflexivity:}} For any node~$i\in V(G)$, consider the sequence of identity maps
\begin{equation}
f^{\ell}_{ii}\,:\,\mathcal{P}_{\ell}(i) \rightarrow \mathcal{P}_{\ell}(i)\nonumber
\end{equation}
and note that this is an adapted sequence, with injective maps that preserve the degree.

\underline{\textbf{Transitivity:}} Let~$\left(h^{\ell}_{ki}\right)$ and~$\left(g^{\ell}_{jk}\right)$ explain~$i\succeq k$ and~$k\succeq j$, respectively, and define the sequence
\begin{equation}
\left(f^{\ell}_{ji}\overset{\Delta}= h^{\ell}_{ki} \circ g^{\ell}_{jk}\right)_{\ell}.\nonumber
\end{equation}
First, note that from Theorem~\ref{th:groupadaptedseq}, the sequence~$\left(f^{\ell}_{ji}\right)$ is an adapted sequence. Since the composition of injective maps is injective, then each~$f^{\ell}_{ji}$ is injective. Now,
\begin{eqnarray}
\mathbf{d}\left(f^{\ell}_{ji}\left(\mathbf{p}^{\ell}\right)\right) & = & \mathbf{d}\left(h^{\ell}_{ki}\left(g^{\ell}_{jk}\left(\mathbf{p}^{\ell}\right)\right)\right)\nonumber\\
& \geq & \mathbf{d}\left(g^{\ell}_{ji}\left(\mathbf{p}^{\ell}\right)\right)\nonumber\\
& \geq & \mathbf{d}\left(\mathbf{p}^{\ell}\right)\nonumber
\end{eqnarray}
and thus,~$i\succeq j$ (explained by~$\left(f^{\ell}_{ji}\right)$).
\end{proof}

Now, we remark that to each preorder, there is an associated equivalence relation defined as
\begin{equation}
i\cong j \Leftrightarrow \left(i\succeq j \wedge j\succeq i \right).\nonumber
\end{equation}
Checking that this is an equivalence relation is trivial. The next Theorem, provides a more explicit characterization for our preorder~$\succeq$.

\begin{theorem}\label{th:characterization}
$i \cong j$, if and only if there exists an adapted sequence~$\left(f^{\ell}_{ji}\right)_{\ell}$ such that
\begin{itemize}
  \item $f^{\ell}_{ji}\,:\,\mathcal{P}_{\ell}(j)\rightarrow \mathcal{P}_{\ell}(i)$ is bijective;
  \item $\mathbf{d}\left(f^{\ell}_{ji}\left(\mathbf{p}^{\ell}\right)\right)=\mathbf{d}\left(\mathbf{p}^{\ell}\right)$;
\end{itemize}
for all~$\mathbf{p}^{\ell} \in \mathcal{P}_{\ell}(j)$ and for all~$\ell$.
\end{theorem}

\begin{proof}
The \textbf{if} part is trivial: it follows as a corollary to Lemma~\ref{lem:inversesequence}.

We are left to prove the \textbf{only if} part. Let~$i\cong j$. Then, by definition,~$i\succeq j$ and~$j\succeq i$. In other words, there exist two adapted sequences~$\left(f_{ji}^{\ell}\right)$ and~$\left(g_{ij}^{\ell}\right)$ that explain~$i\succeq j$ and~$j\succeq i$, respectively. First, note that since~$f_{ji}^{\ell}\,:\,\mathcal{P}_{\ell}(j)\rightarrow \mathcal{P}_{\ell}(i)$ and~$g_{ij}^{\ell}\,:\,\mathcal{P}_{\ell}(i)\rightarrow \mathcal{P}_{\ell}(j)$ are injective maps, then~$\left|\mathcal{P}_{\ell}(i)\right|=\left|\mathcal{P}_{\ell}(j)\right|$ for all~$\ell$ (Schr\"{o}der-Bernstein Theorem) and thus, since the graph is finite, we have~$\left|\mathcal{P}_{\ell}(i)\right|=\left|\mathcal{P}_{\ell}(j)\right|< \infty$, and the injective maps~$f_{ji}^{\ell}\,:\,\mathcal{P}_{\ell}(j)\rightarrow \mathcal{P}_{\ell}(i)$ and~$g_{ij}^{\ell}\,:\,\mathcal{P}_{\ell}(i)\rightarrow \mathcal{P}_{\ell}(j)$ are in fact bijections.

We are to show that~$f^{\ell}_{ji}$ (or~$g^{\ell}_{ij}$) preserves the degree sequence. Define the following interlaced dynamics on~$\mathcal{P}_{\ell}(i)\cup \mathcal{P}_{\ell}(j)$
\begin{itemize}
  \item If $\mathbf{p}^{\ell}\in\mathcal{P}_{\ell}(j)$: $\mathbf{p}^{\ell}\rightarrow f^{\ell}_{ji}\left(\mathbf{p}^{\ell}\right)\rightarrow g^{\ell}_{ij}\left(f^{\ell}_{ji}\left(\mathbf{p}^{\ell}\right)\right)\rightarrow f^{\ell}_{ji}\left(g^{\ell}_{ij}\left(f^{\ell}_{ji}\left(\mathbf{p}^{\ell}\right)\right)\right)\rightarrow \ldots$
  \item If $\mathbf{p}^{\ell}\in\mathcal{P}_{\ell}(i)$: $\mathbf{p}^{\ell}\rightarrow g^{\ell}_{ij}\left(\mathbf{p}^{\ell}\right)\rightarrow f^{\ell}_{ji}\left(g^{\ell}_{ij}\left(\mathbf{p}^{\ell}\right)\right)\rightarrow g^{\ell}_{ij}\left(f^{\ell}_{ji}\left(g^{\ell}_{ij}\left(\mathbf{p}^{\ell}\right)\right)\right)\rightarrow \ldots$
\end{itemize}
We first remark that the vector degree is monotonous over the orbits of the above interlaced dynamical system. Therefore, if~$\mathbf{p}^{\ell}$ is periodic, then
\begin{itemize}
  \item $\mathbf{d}\left(\mathbf{p}^{\ell}\right)=\mathbf{d}\left(f^{\ell}_{ji}\left(\mathbf{p}^{\ell}\right)\right)$, if~$\mathbf{p}^{\ell}\in\mathcal{P}_{\ell}(j)$;
  \item $\mathbf{d}\left(\mathbf{p}^{\ell}\right)=\mathbf{d}\left(g^{\ell}_{ij}\left(\mathbf{p}^{\ell}\right)\right)$, if~$\mathbf{p}^{\ell}\in\mathcal{P}_{\ell}(i)$.
\end{itemize}
In other words, if~$\mathbf{p}^{\ell}\in \mathcal{P}_{\ell}(i)\cup \mathcal{P}_{\ell}(j)$ is a periodic point (or path) w.r.t. the interlaced dynamics, then the degree is preserved by the bijections~$f^{\ell}_{ji}$ and~$g^{\ell}_{ij}$. We claim that any point (or path)~$\mathbf{p}$ is periodic. Indeed, let~$\mathbf{p}\in \mathcal{P}_{\ell}(i)$ and assume that~$\mathbf{p}$ is not periodic. This in particular implies that all the points in the backwards iterates
\begin{equation}
\left(f^{\ell}_{ji}\right)^{-1}(\mathbf{p}),\left(g^{\ell}_{ji}\right)^{-1}\left(\left(f^{\ell}_{ji}\right)^{-1}(\mathbf{p})\right),\ldots\nonumber
\end{equation}
are not periodic as well. In other words, the set of all backwards iterates has infinite cardinality, which is a contradiction as~$\left|\mathcal{P}_{\ell}(i)\cup \mathcal{P}_{\ell}(j)\right|<\infty$. Therefore, we conclude that~$\left(f^{\ell}_{ji}\right)$ is an adapted sequence of bijections preserving the vector degree.
\end{proof}

\begin{theorem}
The equivalence relation~$\cong$ is equal to the coarsest equitable partition (CEP).
\end{theorem}

\begin{proof}
We first show that~$\cong$ conforms to an equitable relation. Then, we show that any other equitable partition in the graph is finer then the coloring of~$\cong$.

\textbf{Part I: $\cong$ is equitable.}

Let~$i\cong j$. Take an arbitrary neighbour~$j_1\in\mathcal{N}(j)$ of~$j$ and let us show that there exists~$i_1\in \mathcal{N}(i)$ with~$j_1\cong i_1$. This establishes the first part.

Indeed, let~$i_1\overset{\Delta}=\pi_{2}\left(f^{1}_{ji}\left(j,j_1\right)\right)$. Our claim is that~$j_1\cong i_1$. Define~$g^{\ell}_{j_1i_1}(\mathbf{p})\overset{\Delta}=\pi_{2\leq \cdot \leq \ell+1}\left(f^{\ell+1}_{ji}(j,\mathbf{p})\right)$ for all~$\mathbf{p}\in \mathcal{P}_{\ell}(j_1)$. We observe that~$\left(g^{\ell}_{j_1i_1}\right)$ fulfills the conditions of Theorem~\ref{th:characterization} and thus,~$j_1\cong i_1$.

\textbf{Part II: $\cong$ is the coarsest equitable relation.}

We will prove that given another equitable relation~$\approx$, then
\begin{equation}
i\approx j \Rightarrow i\cong j.\nonumber
\end{equation}
For each pair~$n,m$ of $\approx$-equivalent nodes~$n\approx m$, define
\begin{equation}
f_{nm}\,:\,\mathcal{N}(n)\rightarrow \mathcal{N}(m)\nonumber
\end{equation}
to be any bijection preserving the $\approx$-classes, i.e., $f_{nm}(k)\approx k$ for all~$k\in\mathcal{N}(n)$. Note that this is possible as~$\approx$ is equitable: $n,m$ have the same degree (hence we can choose a bijection), and have the same number of neighbors per class (hence we can choose a bijection that preserves classes).

Now, assume~$i\approx j$ and define the sequence~$\left(f^{\ell}_{ji}\right)_{\ell}$ by induction as
\begin{equation}
\begin{array}{rcl} f_{ji}^{1}: \mathcal{P}_{1}(j) & \to & \mathcal{P}_{1}(i)\\ \left(j,j_1\right) & \mapsto & \left(i,i_1\right) \end{array}\nonumber
\end{equation}
where~$i_1=f_{ji}(j_1)$, with~$j_1\in \mathcal{N}(j)$ and
\begin{equation}
\begin{array}{rcl} f_{ji}^{n}: \mathcal{P}_{n}(j) & \to & \mathcal{P}_{n}(i)\\ \left(j,j_1,\ldots,j_{n}\right) & \mapsto & \left(i,i_1,\ldots,i_n\right) \end{array}\nonumber
\end{equation}
where,
\begin{equation}
i_1=f_{ji}(j_1),\,i_2=f_{j_{1}i_{1}}(j_2),\,\ldots,i_n=f_{j_{n-1}i_{n-1}}(j_n)\nonumber
\end{equation}
with~$j_1\in \mathcal{N}(j)$, and~$j_k\in\mathcal{N}(j_{k-1})$ for all~$k\leq n$, i.e.,
\begin{equation}
\left(j,j_1,\ldots,j_n\right)\in\mathcal{P}_{n}(j).\nonumber
\end{equation}
Observe first that~$f_{ji}^{n}$ is well constructed as~$j_n\approx i_n$ for each $n$ above -- as one can show by induction -- and thus,~$f_{j_ni_n}$ makes sense. Moreover, the sequence~$\left(f^{\ell}_{ji}\right)$ is an adapted sequence of bijections that preserves the vector degree. By Theorem~\ref{th:characterization}, we conclude that~$i\cong j$.
\end{proof}

The next Theorem provides an alternative practical characterization for the preorder thus constructed: it is an inductive preorder. In fact, in order to prove that the logistic dynamical system preserves the preorder, we need the result in Theorem~\ref{th:induction}.

\begin{theorem}[Induction property of $\succeq$]\label{th:induction}
\begin{equation}\label{eq:inductive}
i\succeq j \Leftrightarrow \exists{f\,:\,\mathcal{N}(j)\rightarrow \mathcal{N}(i) \mbox{ injective:}} f(k)\succeq k\,\, \forall{k\in\mathcal{N}(j)}
\end{equation}
\end{theorem}

\begin{proof}
\textbf{We start by proving the implication~`$\Rightarrow$'.}

Define~$f(k)\overset{\Delta}=\pi_{2}\left(f^{1}_{ji}(j,k)\right)$. Note that~$f\,:\,\mathcal{N}(j)\rightarrow \mathcal{N}(i)$ is injective. Also, define the following sequence
\begin{equation}
g^{\ell}_{kf(k)}\left(k,k_1,k_2,\ldots,k_{\ell}\right)=\pi_{2\leq \cdot \leq \ell+2}\left(f^{\ell+1}_{ji}(j,k,k_1,\ldots,k_{\ell})\right)\nonumber
\end{equation}
and note that~$\left(g^{\ell}_{kf(k)}\,:\, \mathcal{P}_{\ell}(k)\rightarrow \mathcal{P}_{\ell}(f(k))\right)_{\ell}$ is an adapted sequence. Note also that~$g^{\ell}_{kf(k)}\,:\,\mathcal{P}_{\ell}(k)\rightarrow \mathcal{P}_{\ell}(f(k))$ is injective for all~$\ell$. The monotonicity on the vector degree is clearly inherited by~$f^{\ell}_{ji}$. Therefore,~$f(k)\succeq k$.

\textbf{Now, we prove the implication~`$\Leftarrow$'.}

By definition, we have~$f(k)\succeq k$ if and only if there exists an adapted sequence~$\left(g^{\ell}_{kf(k)}\right)$ that explains the inequality. Now, define
\begin{equation}
f^{\ell}_{ji}\left(j, j_1,j_2,\ldots,j_{\ell}\right)=\left(i, g^{\ell-1}_{j_1 f(j_1)}\left(j_1,j_2,\ldots,j_{\ell}\right)\right).
\end{equation}
$\left(f^{\ell}_{ji}\right)$ thus defined is an adapted sequence that explains~$i\succeq j$.
\end{proof}

In fact, our preorder is not only inductive in the sense of equation~\eqref{eq:inductive}, but it is the maximum preorder amongst the inductive preorders as explained next. Let~$\mathcal{O}$ be the set of inductive preorders on a graph~$G$ endowed with the following partial order~$\leq$
\begin{equation}
\widetilde{\succeq}\,\, \leq \,\, \succeq \,\,\Longleftrightarrow \,\,\left(i \,\,\widetilde{\succeq}\,\, j \Rightarrow i \succeq j\right).\nonumber
\end{equation}
It is trivial to check that~$\leq$ is a partial order on~$\mathcal{O}$.

\begin{theorem}
$\succeq$ is the greatest element (a.k.a. maximum) on~$\mathcal{O}$ with respect to the partial order~$\leq$,
\begin{equation}
\widetilde{\succeq}\,\, \leq \,\,\succeq\,\,\forall{\widetilde{\succeq}\in \mathcal{O}},
\end{equation}
i.e., any other preorder~$\widetilde{\succeq}\in \mathcal{O}$ is not only comparable to~$\succeq$, but it is upperbounded by it.
\end{theorem}

\begin{proof}
Let~$\widetilde{\succeq}$ be another inductive preorder. For each pair of nodes~$\left(n,m\right)$ with~$n \,\,\widetilde{\succeq}\,\, m$, choose an injective function
\begin{equation}
f_{mn}\,:\,\mathcal{N}(m)\rightarrow \mathcal{N}(n)\nonumber
\end{equation}
that preserves~$\widetilde{\succeq}$, i.e.,~$f(k)\,\widetilde{\succeq} k$ for all~$k\in \mathcal{N}(m)$. Now, let~$i \widetilde{\succeq} j$ and define the sequence~$\left(f^{\ell}_{ji}\right)_{\ell}$ by induction:
\begin{equation}
\begin{array}{rcl} f_{ji}^{1}: \mathcal{P}_{1}(j) & \to & \mathcal{P}_{1}(i)\\ \left(j,j_1\right) & \mapsto & \left(i,i_1\right) \end{array}\nonumber
\end{equation}
where~$i_1=f_{ji}(j_1)$, with~$j_1\in \mathcal{N}(j)$ and
\begin{equation}
\begin{array}{rcl} f_{ji}^{n}: \mathcal{P}_{n}(j) & \to & \mathcal{P}_{n}(i)\\ \left(j,j_1,\ldots,j_{n}\right) & \mapsto & \left(i,i_1,\ldots,i_n\right) \end{array}\nonumber
\end{equation}
where,
\begin{equation}
i_1=f_{ji}(j_1),\,i_2=f_{j_{1}i_{1}}(j_2),\,\ldots,i_n=f_{j_{n-1}i_{n-1}}(j_n)\nonumber
\end{equation}
with~$j_1\in \mathcal{N}(j)$, and~$j_k\in\mathcal{N}(j_{k-1})$ for all~$k\leq n$, i.e.,
\begin{equation}
\left(j,j_1,\ldots,j_n\right)\in\mathcal{P}_{n}(j).\nonumber
\end{equation}
Note that~$f^{\ell}_{ji}$ is well defined. Also, each~$f^{\ell}_{ji}$ is injective and preserves the degree sequence. Moreover, the sequence~$\left(f^{\ell}_{ji}\right)_{\ell}$ is adapted.
\end{proof}

\section{Logistic dynamics preserves the preorder}\label{sec:preserva}

In this section, we show that the logistic dynamics~\eqref{eq:logistic} preserves the preorder~$\succeq$. As a result, this leads to the emergence of a non-trivial invariant set for this type of dynamical system over networks, whose characterization is tied to the underlying graph structure.

\begin{theorem}\label{th:preservesineq}
Let~$\mathbf{y}(0)\in\left[0,1\right]^{N}$ be so that
\begin{equation}
y_{i}(0)\geq y_{j}(0)\,\forall{i\succeq j}.\label{eq:assumption}
\end{equation}
Then,
\begin{equation}
y_{i}(t,\mathbf{y}(0))\geq y_{j}\left(t,\mathbf{y}(0)\right)\,\,\forall{t\geq 0}.
\end{equation}
\end{theorem}
Equivalently, the set
\begin{equation}
S\overset{\Delta}=\left\{\mathbf{y}\in\left[0,1\right]^{N} \,:\, y_i\geq y_k,\,\forall{i\succeq k} \right\}\label{eq:invariant}
\end{equation}
is invariant to the logistic dynamical system.

\begin{proof}

To show that~$S$ (defined in equation~\eqref{eq:invariant}) is invariant under the logistic dynamics (either continuous or discrete-time), we just need to check the qualitative behavior of the trajectory~$\left(\mathbf{y}(t)\right)$ at the border of~$S$ to establish that it does not escape from the set. We focus attention on the continuous time case.

Let~$i\succeq j$ with~$y_i(s)=y_j(s)$, for some time~$s$, i.e., the trajectory is at the border at time~$s$. Let the adaptive sequence~$\left(f_{ji}^{\ell}\right)$ explain the inequality~$i\succeq j$. For simplicity, denote~$f_{ji}:=f_{ji}^{1}$ with
\begin{equation}
f_{ji}\,:\,\mathcal{N}(j)\rightarrow \mathcal{N}(i).\nonumber
\end{equation}
Assume first that
\begin{equation}
y_{k}(s) > y_{f_{ji}(k)}(s)
\end{equation}
for some neighbor~$k\in\mathcal{N}(i)$. Given the induction property of the preorder in Theorem~\ref{th:induction} and the assumption in equation~\eqref{eq:assumption}, we conclude, by inspection on the ODE equation~\eqref{eq:logistic}, that
\begin{equation}
\overset{\cdot}y_i(s)>\overset{\cdot} y_j(s)\nonumber
\end{equation}
and thus, there exists~$\delta>0$ such that~$y_i(t)>y_j(t)$ for all $t\in\left.\left[s,s+\delta\right.\right)$. In other words, the trajectory does not scape the set.

In the case where the two nodes have the same configuration in a first order neighborhood, i.e.,~$f_{ji}$ is bijective and
\begin{equation}
y_k(s)=y_{f_{ji}(k)}(s)
\end{equation}
for all~$k\in \mathcal{N}(i)$, then the previous argument does not apply since~$\overset{\cdot}y_i(s)=\overset{\cdot}y_j(s)$ and we need to consider higher order derivatives. The equations for the $(n+1)$th order derivatives of the logistic system are given by
\begin{eqnarray}
\overset{(n+1)}y_i(t) & = & \gamma \sum_{k\sim i} \left[\overset{(n)}y_k(t)\left(1-y_i(t)\right)- \sum_{m=0}^{n-1}\left(\begin{array}{c} n \\ m \end{array}\right)\overset{(m)}y_k(t) \overset{(n-m)}y_i(t)\right]-\overset{(n)}y_i(t)\label{eq:high}\\
\overset{(n+1)}y_j(t) & = & \gamma \sum_{\ell\sim j} \left[\overset{(n)}y_{\ell}(t)\left(1-y_j(t)\right)- \sum_{m=0}^{n-1}\left(\begin{array}{c} n \\ m \end{array}\right)\overset{(m)}y_{\ell}(t) \overset{(n-m)}y_j(t)\right]-\overset{(n)}y_j(t).\label{eq:high2}
\end{eqnarray}
Given a vector of indexes~$I=\left(I_1,I_2,\ldots,I_N\right)\in\mathbb{N}^N$, define~$y_I:=\left(y_{I_1},\ldots,y_{I_N}\right)$. Assume, for some~$N\geq 1$, that the adapted sequence~$\left(f_{ji}^{\ell}\right)$ is such that~$f_{ji}^{m}\,:\,\mathcal{P}_{m}(j)\rightarrow \mathcal{P}_{m}(i)$ is bijective for all~$m\leq N$ and
\begin{equation}
y_{p^{m}}=y_{f^m_{ji}\left(p^{m}\right)},
\end{equation}
for all~$p^m\in \mathcal{P}^m(j)$, and~$m\leq N$, but at~$N+1$ we have (the strict dominance)
\begin{equation}
f_{ji}^{N+1}\mbox{ injective and not bijective, or } y_{p_{N+1}^{N+1}}<y_{f_{ji}\left(p^{N+1}_{N+1}\right)}.\nonumber
\end{equation}
Then, by inspection on the higher order derivative equations~\eqref{eq:high}-\eqref{eq:high2}, and applying induction on~$N$, we conclude that
\begin{eqnarray}
\overset{(n)}y_i(s) & = & \overset{(n)}y_j(s),\,\forall{n\leq N}\\
\overset{(N+1)}y_i(s) & > & \overset{(N+1)}y_j(s)
\end{eqnarray}
and, thus, there exists~$\delta>0$ such that~$y_i(t)>y_j(t)$ for all $t\in\left.\left[s,s+\delta\right.\right)$, that is, the trajectory does not escape the set~$S$.


\end{proof}

\begin{corollary}[LD preserves CEP]\label{co:preservesCEP}
The homogeneous logistic dynamical system preserves the CEP of the graph, i.e., if~~$\mathbf{y}(0)\in\left[0,1\right]^{N}$ is so that
\begin{equation}
y_{i}(0)= y_{j}(0)\,\forall{i\cong j}.
\end{equation}
Then,
\begin{equation}
y_{i}(t,\mathbf{y}(0))= y_{j}\left(t,\mathbf{y}(0)\right)\,\,\forall{t\geq 0}\,\,\,\,\,\forall{i\cong j}.
\end{equation}
\end{corollary}

Corollary~\ref{co:preservesCEP} follows from Theorems~\ref{th:characterization} and~\ref{th:preservesineq}.

\textbf{Remark.} The proof of Theorem~\ref{th:preservesineq} naturally extends to establish the result for a broader class of dynamical systems. In particular, let~$\overline{F}\,:\,\mathbb{R}\times\mathbb{R}^{N-1}\rightarrow \mathbb{R}$ be an analytic real valued function invariant under permutations on the second vector coordinate, i.e.,
\begin{equation}
\overline{F}\left(x_1;x_2,\ldots,x_N\right)= \overline{F}\left(x_1;p\left(x_2,\ldots,x_N\right)\right)
\end{equation}
for all permutations~$p\,:\,\mathbb{R}^{N-1}\rightarrow \mathbb{R}^{N-1}$ and for all~$\mathbf{x}=\left(x_1,\ldots,x_N\right)\in\mathbb{R}^{N}$. Assume also that~$\overline{F}$ is monotonous in the second vector coordinate, i.e.,
\begin{equation}
\mathbf{x} \geq \mathbf{y}\in\mathbb{R}^{N-1}\Rightarrow \overline{F}\left(x_1;\mathbf{x}\right)\geq \overline{F}\left(x_1;\mathbf{y}\right).
\end{equation}
Given a graph~$G$, let~$F_G=\left(F_1,F_2,\ldots,F_N\right)$ be the vector field induced by~$\overline{F}$ as follows
\begin{equation}\label{eq:construct2}
F_i\left(x\right)= \overline{F}\left(x_i, x_{i_1},\ldots,x_{i_k},0,\ldots,0\right)
\end{equation}
where,~$\mathcal{N}(i)\overset{\Delta}=\left\{j\sim i\,:\, j\in V\right\}=\left\{i_1,\ldots,i_k\right\}$ is the set of neighbors to~$i$ in the graph~$G$ (i.e.,~$F$ conveys the graph structure of~$G$). Note that the ordering of~$x_{i_1},\ldots,x_{i_k}$ in equation~\eqref{eq:construct2} is not relevant as~$\overline{F}$ is invariant under permutation of the associated coordinates. By evoking Fa\`{a} di Bruno's formula (refer, e.g., to~\cite{faadi}) for higher order derivatives of~$F_G$ and resorting to a proof similar to the one in Theorem~\ref{th:preservesineq}, one can establish the result in Theorem~\ref{th:preservesineq} to the class of dynamical systems
\begin{equation}
\frac{d}{dt}\mathbf{y}(t)=F_G(\mathbf{y}(t))
\end{equation}
with vector field~$F_G$. Note that~$\mathbf{F}$ is analytic and, thus, there exists unique solution~$\left(\mathbf{y}\left(t,\mathbf{y}(0)\right)\right)$ for each initial condition~$\mathbf{y}(0)\in\mathbb{R}^N$.

\section{Concluding remarks}\label{sec:conclusion}

In this paper, we constructed a preorder on the set of nodes of a graph. Such preorder entails the classical coarsest equitable partition of the graph. Moreover, it is preserved by certain monotonous dynamical systems over networks, which leads to the existence of an invariant set whose characterization relies on the graph structure of the dynamical system. Note that in light of Corollary~\ref{co:preservesCEP}, such dynamical systems admit a lower dimensional version with underlying dimension given by the number of colors associated with the CEP of the graph, i.e., by having the same initial conditions for nodes with same color, the state of such nodes will be synchronized for all time. Also, Theorem~\ref{th:preservesineq} tells us that one can lower-bound and/or upper-bound solutions to such dynamical systems by properly lower-bounding and/or upper-bounding the corresponding initial conditions. These observations provide with an application to approximate such dynamical systems over complex networks by lower-dimensional ones: i) find the CEP of the graph (which has a cost of about~$n \log n$, with~$n$ being the number of nodes); and ii) initialize evenly the state of nodes with same color so to lower-bound and/or upper-bound the arbitrary initial conditions of interest. In this case, the solutions (with arbitrary initial conditions) can be lower-bounded and/or upper-bounded by its lower dimensional versions for all time~$t$.

\small
\clearpage
\bibliographystyle{IEEEtran}
\bibliography{IEEEabrv,biblio}

\end{document}